\documentclass[12pt]{amsproc}
\usepackage{fullpage}
\usepackage{latexsym}
\usepackage{amscd, amsfonts, eucal, mathrsfs, amsmath, amssymb, amsthm}
\input xy
\xyoption{all}

\usepackage{pdfsync}
\usepackage{baskervald}
\usepackage{comment}

\usepackage[OT2,T1]{fontenc}
\DeclareSymbolFont{cyrletters}{OT2}{wncyr}{m}{n}
\DeclareMathSymbol{\Sha}{\mathalpha}{cyrletters}{"58}

\newcommand{\field}[1]{\mathbf #1}

\newcommand{\mc}[1]{\mathcal #1}
\newcommand{\ms}[1]{\mathscr #1}
\newcommand{\widebar}[1]{\overline{#1}}

\newcommand{\R}{\field R}

\newcommand{\C}{\field C}

\newcommand{\Z}{\field Z}
\newcommand{\Q}{\field Q}
\newcommand{\N}{\field N}

\newcommand{\simto}{\stackrel{\sim}{\to}}

\newcommand{\send}{\ms E\!nd}

\newcommand{\lf}{\text{\rm lf}}
\newcommand{\tf}{\text{\rm tf}}
\newcommand{\stable}{\text{\rm s}}
\newcommand{\semistable}{\text{\rm ss}}

\DeclareMathOperator{\Spec}{Spec}

\DeclareMathOperator{\AHilb}{AHilb}
\DeclareMathOperator{\Inertia}{I}
\newcommand{\inertia}{\iota}
\DeclareMathOperator{\Site}{Site}
\newcommand{\simplespace}{merbe}
\DeclareMathOperator{\Can}{Can}
\DeclareMathOperator{\Contr}{Contr}
\DeclareMathOperator{\Mod}{Mod}
\DeclareMathOperator{\mSh}{\bf Sh}
\DeclareMathOperator{\cSh}{\widetilde{\mSh}}
\DeclareMathOperator{\Types}{\bf Level}
\DeclareMathOperator{\GrpSch}{\bf GrpSch}
\DeclareMathOperator{\SSp}{\bf Merbe}
\newcommand{\perf}{\text{\rm perf}}
\DeclareMathOperator{\Qcoh}{QCoh}

\renewcommand{\phi}{\varphi}

\renewcommand{\P}{\field P}

\newcommand{\A}{\field A}

\DeclareMathOperator{\Pic}{Pic}
\DeclareMathOperator{\Sh}{Sh}
\DeclareMathOperator{\Quot}{Quot}

\DeclareMathOperator{\D}{D}

\DeclareMathOperator{\pr}{pr}

\DeclareMathOperator{\GL}{GL}
\DeclareMathOperator{\SL}{SL}

\DeclareMathOperator{\Ext}{Ext}

\DeclareMathOperator{\PGL}{PGL}

\newcommand{\m}{\boldsymbol{\mu}}

\newcommand{\G}{\field G} 

\renewcommand{\H}{\operatorname{H}}

\newcommand{\Gal}{\operatorname{Gal}}

\DeclareMathOperator{\per}{per}
\DeclareMathOperator{\ind}{ind}

\DeclareMathOperator{\gr}{gr}

\DeclareMathOperator*{\tensor}{\otimes}

\DeclareMathOperator{\rk}{\operatorname{rk}}

\newcommand{\inj}{\hookrightarrow}

\newcommand{\id}{\operatorname{id}}

\DeclareMathOperator{\Isom}{\operatorname{Isom}}
\DeclareMathOperator{\Hom}{\operatorname{Hom}}

\DeclareMathOperator{\M}{\operatorname{M}}

\DeclareMathOperator{\Br}{\operatorname{Br}}

\DeclareMathOperator{\B}{\operatorname{B\!}}

\newcommand{\Sch}{\mathbf{Sch}}
\newcommand{\AffSch}{\mathbf{AffSch}}
\newcommand{\Set}{\mathbf{Set}}
\DeclareMathOperator{\Groupoid}{\operatorname{\bf Grpd}}

\DeclareMathOperator{\Jac}{Jac}

\newcommand{\eps}{\varepsilon}

\newcommand{\Chow}{\operatorname{CH}}

\newtheorem{lem}{Lemma}[section]

\renewcommand{\thelem}{\ifnum\value{subsubsection}>0{\thesubsubsection.\arabic{lem}}\else{\ifnum\value{subsection}>0{\thesubsection.\arabic{lem}}\else{\thesection.\arabic{lem}}\fi}\fi}

\newtheorem{thm}[lem]{Theorem}

\newtheorem{prop}[lem]{Proposition}

\newtheorem{cor}[lem]{Corollary}

\theoremstyle{definition}
\newtheorem{defn}[lem]{Definition}

\newtheorem{example}[lem]{Example}

\theoremstyle{remark}
\newtheorem{remark}[lem]{Remark}

\newtheorem{notation}[lem]{Notation}

\newtheorem{question}[lem]{Question}
\newtheorem{assumption}[lem]{Assumption}

\numberwithin{equation}{subsection}

\author{Max Lieblich}

\title{Moduli of sheaves: a modern primer}

\begin{document}
\maketitle

\tableofcontents

\section{Introduction}
\label{sec:intro}

This paper is focused on results coming from the moduli spaces of sheaves,
broadly understood, over the last 15 years or so. The subject has seen an
interesting convergence of mathematical physics, derived categories, number
theory, non-commutative algebra, and the theory of stacks. It has spun off
results that inform our understanding of the Hodge conjecture, the
Tate conjecture, local-to-global principles, the structural properties
of division
algebras, the $u$-invariants of quadratic forms over function fields, and the
geometry of moduli spaces of K3 surfaces, among other things. All of these
results turn out to be interconnected in fascinating ways. For
example, the earliest results on derived categories of elliptic
threefolds using ``non-fine moduli spaces of sheaves'' (as studied by
C\u{a}ld\u{a}raru \cite{MR1887894}) are very closely
related to subsequent attacks, nearly a decade later, on the Tate conjecture and
the geometry of the Ogus space of supersingular K3 surfaces, where the
non-fineness is turned around to spawn new surfaces from Brauer classes.

The most fascinating thing about these developments is that they rely
heavily on porting chunks of the classical GIT-driven theory of moduli
of sheaves into the domain of so-called ``twisted sheaves''. Thus, for
example, unirationality results for moduli spaces of vector bundles on
curves, originally due to Serre, interact with the Graber-Harris-Starr
theorem (and its positive characteristic refinement by Starr and de
Jong) to produce similar kinds of estimates for function fields of
surfaces over algebraically closed fields.  Similarly, irreducibility
results dating back to Gieseker, Li, and O'Grady in the classical
case, themselves inspired by the work of Donaldson, Mrowka and Taubes
in the analytic category, can be proven for twisted sheaves and
combined with finite-field methods and the aforementioned results over
algebraically closed fields to yield explicit bounds on the rank of
division algebras over function fields of surfaces over finite fields.

Spending enough time with these ideas, one is naturally led to the observation
that the central function of these moduli spaces is to create a flow of
information through correspondences created by the universal sheaf. (This is
laid most bare in the study of derived equivalences of twisted sheaves. Indeed,
in some sense the derived equivalences are the richest shadows of underlying
equivalences of various kinds of motives, something we will not discuss here at
all.) It then becomes apparent that the successes arising from the theory of
``twisted sheaves'' are primarily psychological: by viewing the moduli stack as
its own object (rather than something to be compressed into a scheme), one more
easily retains information that can be flowed back to the land of varieties.

This hints that a more uniform theory -- one that treats sheaves and twisted
sheaves as objects of an identical flavor -- may be useful for future study of
these objects. Among other things, we endeavor to describe such a theory here.

\subsection{The structure of this paper}
\label{sec:structure}

This paper has two primary purposes: (1) describing some background and recent
results, and (2) doing a thought experiment about the right way to set up a more
coherent theory that incorporates both sheaves and twisted sheaves as equals.
The first purpose is addressed in Part \ref{part:background} and the second is
attempted in Part \ref{part:thought-experiment}.

Part \ref{part:background} includes a brief tour of the classical theory of
sheaves in Section \ref{sec:classical}, including key examples and a few remarks
about stability, and some less classical examples in Section
\ref{sec:motivations} showing the usefulness of ``thinking in stacks'' from the
beginning. Part \ref{part:background} concludes with an impressionistic
catalogue of results proven using twisted sheaves in the last decade or so.

Part \ref{part:thought-experiment} attempts to build a theory where ambient
space and moduli space are treated symmetrically, and universal sheaves always
exist. After establishing some terminology in Section \ref{sec:terminology}, we
describe the resulting theory of moduli of sheaves in Section \ref{sec:basics},
and some brief case studies illustrating how it can be used in Section \ref{sec:case-studies}.

\subsection{Background assumed of the reader}
\label{sec:background}
I assume that the reader is familiar with schemes, algebraic spaces, and stacks. 
I also assume that she has seen the Picard scheme and the $\Quot$ scheme. While
it is almost certainly the case that a reader who knows about algebraic stacks
has also seen moduli spaces of sheaves to some extent, in the modern world there
are many students of stacks who have little or no exposure to the moduli theory
of sheaves, classical or modern. Thus, I will tell some of that story
from scratch here.

\part{Background}\label{part:background}

\section{A mild approach to the classical theory}
\label{sec:classical}

We start with a chunk of the classical theory, to give a reader the flavor of
the geometry of the moduli spaces of sheaves. In Section \ref{sec:basics}
we will give a more rigorous development of the theory in a more general
framework, so we content ourselves here to discuss the beautiful geometry
without worrying too much about the details.

\subsection{The $\Quot$ scheme}
\label{sec:quot}

The first moduli problem that many people encounter is a moduli problem of
diagrams of sheaves: the $\Quot$ functor. Fix a morphism $f:X\to S$ of schemes
that is locally of finite presentation and a quasi-coherent sheaf of finite
presentation $F$ on $X$.

\begin{defn}
  The \emph{quotient functor associated to $F$\/} is the functor
  $$\Quot_F:\Sch_S\to\Set$$
  that sends $T\to S$ to the set of isomorphism classes of quotients
  $$F_T\to Q,$$ with $Q$ a $T$-flat finitely presented quasi-coherent sheaf on $X_T$
  whose support is proper and of finite presentation over $T$. 
\end{defn}
An isomorphism of quotients $F\to Q_1$ and $F\to Q_2$ is an isomorphism $Q_1\to
Q_2$ commuting with the maps from $F$. One can check that if such an isomorphism
exists then it is unique.

\begin{example}
  The simplest example of $\Quot_F$ is the following: let $X\to S$ be
  $\Spec\Z\to\Spec\Z$ and let $F$ be $\ms O_{\Spec\Z}^{\oplus n}$. Then
  $\Quot_F$ is the disjoint union of all of the Grassmannians of an $n$-dimensional
  vector space (over $\Spec\Z$).
\end{example}

\begin{example}
  The other standard example (and the original motivation for studying the
  $\Quot$ functor) is the following: suppose $X\to S$ is separated
  and $F=\ms O_X$. In this case, $\Quot_F$ is the Hilbert scheme of $X$. Among
  many other incredibly complex components, it contains one component isomorphic
  to $X$ itself: one can look at those quotients $\ms O_{X_T}\to Q$ where $Q$ is
  the structure sheaf of a section $T\to X_T$ of the canonical projection
  $X_T\to T$. (Here we use the fact that $X\to S$ is separated to know that $Q$
  has proper support.)
\end{example}

In particular, we see that $\Quot_F$ can have arbitary geometry: many
components with arbitrary dimensions and Kodaira dimensions.

The fundamental theorem about $\Quot_F$ is that it is representable. (Of course,
the representing object depends heavily on the context; it may be an algebraic
space.)

\begin{thm}[Grothendieck]\label{T:groth-quot}
  If $X\to S$ is locally quasi-projective and locally of finite presentation
  over $S$, then $\Quot_F$ is representable by a scheme that is locally of
  finite presentation and locally quasi-projective over $S$.
\end{thm}

Much of the classical theory of moduli spaces of sheaves is built by
bootstrapping from Theorem \ref{T:groth-quot}, as we will indicate below. The
general idea is to rigidify a problem by adding some kind of additional
structure (e.g., a basis for the space of sections of a large twist),
and then realize the original as a quotient of a locally closed subscheme of
some $\Quot$ scheme.

One question that is naturally raised by the $\Quot$ functor, and more
particularly by the Hilbert scheme, is the following.

\begin{question}
  What happens if we try to make the Hilbert scheme without the quotient
  structure? I.e., why not take moduli of sheaves $L$ on $X$ that are
  invertible sheaves supported on closed subschemes of $X$? How does this moduli
  problem compare to the Hilbert scheme?
\end{question}

\subsection{The Picard scheme}
\label{sec:picard}

The next moduli space of sheaves that we learn about is the Picard scheme. Fix a
proper morphism $f:X\to S$ of finite presentation between schemes.

\begin{defn}
  The \emph{naive Picard functor of $X/S$\/} is the functor
  $$\Pic_{X/S}^{\text{\rm naive}}:\Sch_S\to\Set$$
  sending $T\to S$ to $\Pic(X_T)$.
\end{defn}

This can't possibly be represented by a scheme, which one can see by considering
the identity map $S\to S$. The functor $T\mapsto\Pic(T)$ has the property that
for all $T$ and all $f\in\Pic(T)$ there is a Zariski covering $U\to T$ such that
$f$ maps to $0$ in $\Pic(U)$. That is, the functor defined is not a sheaf (it is
not even a separated presheaf in the Zariski topology) and thus cannot be a scheme.

The traditional way to fix this is to study the sheafification of
$\Pic^{\text{\rm naive}}_{X/S}$. Once one has studied Grothendieck
topologies, one recognizes that one is in fact studying the higher direct image 
$\R^1f_\ast\G_m$ in the fppf topology. This functor is usually written $\Pic_{X/S}$.

\begin{thm}[Grothendieck]
  If $X\to S$ is a proper morphism of schemes of finite presentation that is
  locally projective on $S$ and cohomologically flat in dimension $0$
  then $\Pic_{X/S}$ is representable by a scheme locally of finite presentation over
  $S$.
\end{thm}
This is beautifully described in Chapter 8 of \cite{MR1045822}. (The
definition of ``cohomologically flat in dimension 0'' is on page 206,
in the paragraph preceding Theorem 7 of that chapter.)

\begin{example}
  Everyone's first example is the Picard scheme of $\P^1$. By cohomology and
  base change, for any scheme $T$ we have a canonical isomorphism
  $$\Pic(\P^1\times T)\simto\Pic(\P^1)\times\Pic(T)=\Z\times\Pic(T).$$
  Since the second factor is annihilated upon sheafification, we see that
  $$\Pic_{\P^1/\Spec\Z}\cong\underline{\Z}_{\Spec\Z},$$
  the constant group scheme with fiber $\Z$.
\end{example}

In Section \ref{sec:motivating-example} we will elaborate on this example for
non-split conics.

\subsection{Sheaves on a curve}
\label{sec:curve}

General locally free sheaves on a curve turn out to have nice moduli. Fix a
smooth proper curve $C$ over an algebraically closed field $k$, and fix an
invertible sheaf $L$ on $C$.

\begin{defn}
  The \emph{stack of locally free sheaves on $C$ of rank $n$ and determinant
$L$\/}, denoted $\mSh_{C/k}^\lf(n,L)$, is the stack whose objects over a
$k$-scheme $T$ are pairs $(F,\phi)$ where $F$ is a locally free sheaf on
$C\times T$ of rank $n$ and $\phi:\det F\simto L_T$ is an identification of the
determinant of $F$ with $L$.
\end{defn}

The primary geometric result about $\mSh_{C/k}^\lf(n, L)$ is the following.

\begin{prop}\label{P:unirational-curve-classical}
  The stack $\mSh^\lf_{C/k}(n,L)$ is an integral algebraic stack that is an ascending
  union of open substacks $\ms U_N$, each of which admits a surjection from an
  affine space $\A^{n(N)}$. That is, $\mSh^\lf_{C/k}(n,L)$ is
  an ascending union of unirational open substacks.
\end{prop}

There are several ways to prove Proposition \ref{P:unirational-curve-classical}. Here is one of
them. First, let $F$ be a locally free sheaf of rank $n$ and determinant $L$.

\begin{lem}\label{L:bertini-sketch}
  For sufficiently large $m$, a general map $\ms O^{n-1}\to F(m)$ has invertible
  cokernel isomorphic to $\det(F(m))\cong\det(F)(nm)$.
\end{lem}
\begin{proof}[Sketch of proof]
  This is a Bertini-type argument. Choose $m$ large enough that for every point
  $c$ of $C$ the restriction map $$\Gamma(C,F(m))\to F(m)\tensor\kappa(c)$$ is
  surjective. Let $A$ be the affine space underlying $\Hom(\ms O^{n-1}, F(m))$,
  and let $$\Phi:\ms O^{n-1}\to\pr_2^\ast F(m)$$ on $A\times C$ be the universal
  map. There is a well-defined closed locus $Z\subset A\times C$ parametrizing points where
  the fiber of $\Phi$ does not have full rank. As with the Bertini theorem, we
  want to show that the codimension of $Z$ is at least $2$, as then it cannot
  dominate $A$, giving the desired map.

  To show it has codimension at least $2$, it suffices to show this in each
  fiber over a point $c\in C$. By assumption, the universal map specializes to
  all maps $\kappa(c)^{n-1}\to F(m)\tensor\kappa(c)$, so it suffices to show
  that the locus of maps in the linear space $\Hom(k^{n-1}, k^n)$ that have
  non-maximal rank is of codimension at least $2$. One way to do this is to show
  that the determinantal variety cut out by the $(n-1)\times(n-1)$-minors of the
  $(n-1)\times n$-matrices over $k$ is not a divisor (i.e., there are
  non-redundant relations coming from the determinants). Another quick and dirty
  way to see the desired inequality in this case is to note that any map of
  non-maximal rank must factor through the quotient by a line. The space of
  quotients has dimension $n-2$, while the maps from the quotient vector space
  have 
  dimension $n(n-2)$. Thus, the space of non-injective maps has dimension at
  most $(n+1)(n-2)=n^2-n-2$, but the ambient space has dimension $n(n-1)=n^2-n$.
  The codimension is thus at least $2$, as desired.

  The value of the cokernel arises from a computation of the determinant of the
  exact sequence of locally free sheaves resulting from a general map $\ms
  O^{n-1}\to F(m)$.
\end{proof}

\begin{proof}[Proof of Proposition \ref{P:unirational-curve-classical}]
  We first address the latter part of the Proposition, deferring algebraicity of
  the stack for a moment.
  
  Given $m$, define a functor $$e_m:\AffSch_k\to\Set$$ on the category of affine
  $k$-schemes that sends $T$ to
  $$\Ext_{C\times T}^1(\pr_1^\ast L(mn),\ms O^{n-1}).$$
  By cohomology and base change, for all sufficiently large $m$, the functor
  $e_m$ is representable by an affine space $\A^N$. The identity map defines a
  universal extension
  $$0\to \ms O^{n-1}\to \ms F\to\pr_1^\ast L(mn)\to 0$$
  over $C\times\A^N$. By taking determinants, this comes with an isomorphism
  $$\phi:\det(\ms F)\simto \pr_1^\ast L(mn).$$ Sending this universal extension
  to $(\ms F(-m), \phi(-mn))$ defines a
  morphism $$\eps_m:\A^N\to\mSh_{C/k}^\lf(n,L).$$
  By Lemma \ref{L:bertini-sketch},
  these maps surject onto $\mSh_{C/k}^\lf(n,L)$ (as $m$ ranges over any given
  unbounded collection of integers $m$ for which $e_m$ is representable).

 Let us briefly address algebraicity of the stack. First, we note that for large
 enough $m$, the morphisms $\eps_m$ are smooth over a dense open subset of $\A^N$. This is equivalent to
 the statement that the sections of $F(m)$ lift under any infinitesimal
 deformation (for $F$ in the image), and this follows precisely from the vanishing of $\H^1(C,F(m))$,
 which happens over a dense open of $\A^N$ for sufficiently large $m$. Moreover,
 for any $F$, there is some $m$ such that the smooth locus of $\eps_m$ contains
 $F$. The
 diagonal of $\mSh^{\lf}_{C/k}(n,L)$ is given by the $\Isom$ functor, which we
 know to be representable by Grothendieck's work on cohomology and base change.
 We conclude that we have defined a representable smooth cover of
 $\mSh^\lf_{C/k}(n,L)$, showing that it is algebraic. 

 It is also possible to prove algbraicity using Artin's theorem and studying the
 deformation theory of sheaves, which is particularly simple for locally free
 sheaves on a curve. In fact, there are never any obstructions, showing that the
 stack $\mSh^\lf_{C/k}(n,L)$ is itself smooth.
\end{proof}

We have studied a completely general moduli problem, but the reader will usually
encounter an open substack of $\mSh^\lf_{C/k}(n,L)$ in the literature: the stack
of \emph{semistable\/} sheaves. What's more, the classical literature usually
addresses the Geometric Invariant Theory (GIT) quotient of this stack.
For the sake of quasi-completeness, we give the definitions
and make a few remarks. We will not discuss GIT here at all. The reader is
referred to \cite{MR2665168} (or the original \cite{MR1304906}) for further information.

\begin{defn}
  A locally free sheaf $V$ on a curve $C$ is \emph{stable\/} (respectively
  \emph{semistable\/}) if for all nonzero subsheaves $W\subset V$ we have the inequality
  $$\frac{\deg(W)}{\rk(W)} < \frac{\deg(V)}{\rk(V)}$$
  (respectively,
  $$\frac{\deg(W)}{\rk(W)} \leq \frac{\deg(V)}{\rk(V)}.)$$
\end{defn}

The quantity $\deg(F)/\rk(F)$ is usually called the \emph{slope\/}, written $\mu(F)$.

After Mumford revived the theory of stable bundles,
this was studied by Narasimhan and Seshadri \cite{MR0184252}, who were
interested in unitary bundles over a curve. It turns out that stable sheaves of degree $0$ and
rank $n$ over a curve correspond to irreducible unitary representations of the
fundamental group of $C$; this was the original target of study for Narasimhan
and Seshadri.

Stable and semistable sheaves have a number of nice categorical properties.
Stable sheaves are simple (that is, a map $F\to G$ is either an isomorphism or
$0$). Semistable sheaves $H$ admit Jordan-H\"older filtrations
$$H=H^0\supset H^1\supset\cdots\supset H^n=0$$ such that $H^i/H^{i+1}$ is stable
with the same slope as $H$, and the associated graded $\gr(H^\bullet)$ is
uniquely determined by $H$ (even though the filtration itself is not). We will
write this common graded sheaf as $\gr(H)$ (omitting the filtration from the
notation). This is all nicely summarized in \cite{MR2665168}.

The basic geometric result concerning stability is the following. We can form a
stack of stable sheaves $\mSh_{C/k}^\stable(n,L)$ (resp.\ semistable sheaves
$\mSh_{C/k}^\semistable(n,L)$) inside $\mSh_{C/k}^\lf(n,L)$ by requiring that the
geometric fibers in the family are stable (resp.\ semistable).

\begin{prop}
  The inclusions
  $$\mSh^\stable_{C/k}(n,L)\to\mSh^\semistable_{C/k}(n,L)\to\mSh^\lf_{C/k}(n,L)$$
  are open immersions.
\end{prop}
\begin{proof}[Idea of proof]
  Everything in sight is locally of finite presentation, and the arrows in the
  diagram are of finite presentation, hence have constructible images. It thus
  suffices to show that the images are stable under generization. In other
  words, we may restrict ourselves to families of sheaves on $C\tensor R$ where
  $R$ is a complete dvr. Given a locally free sheaf $V$ on $C\tensor R$ and a
  subsheaf $W^\circ\subset V_\eta$ of the generic fiber such that
  $\mu(W^\circ)\geq\mu(V_\eta)$, then we can extend $W^\circ$ to a subsheaf
  $W\subset V$ such that $V/W$ is flat over $R$. This means that the sheaf $W$
  is locally free on $C\tensor R$. But any locally free sheaf has constant
  degree in fibers, so the specialization $W_0$ also satsifies
  $\mu(W_0)\geq\mu(V_0).$ This shows the desired stability under generization
  (by showing the complement is closed under specialization).
\end{proof}

One of the early successes of GIT was a description
of a scheme-theoretic avatar of $\mSh^\semistable_{C/k}(n,L)$.

\begin{thm}[Mumford, modern interpretation]\label{sec:sheaves-curve}
  There is a diagram
  $$\xymatrix{\mSh^\stable_{C/k}(n,L)\ar[r]\ar[d]_\pi &
    \mSh^\semistable_{C/k}(n,L)\ar[d]^{\pi'}\\
    \cSh^\stable_{C/k}(n,L)\ar[r] & \cSh^\semistable_{C/k}(n,L),}$$
  in which
  \begin{enumerate}
  \item $\cSh^\semistable_{C/k}(n,L)$ is a locally factorial integral
    projective scheme;
  \item $\cSh^\stable_{C/k}(n,L)$ is smooth;
    \item $\cSh^\stable_{C/k}(n,L)\to \cSh^\semistable_{C/k}(n,L)$ is an open
      immersion;
      \item $\pi$ together with the scalar multiplication structure realize
  $\mSh^\stable_{C/k}(n,L)$ as a $\m_n$-gerbe over $\cSh^\stable_{C/k}(n,L)$.
  \end{enumerate}
  Moreover,
  \begin{enumerate}
  \item $\pi'$ universal among all morphisms from
    $\mSh^\semistable_{C/k}(n,L)$ to schemes;
  \item $\pi$ establishes a bijection between
    the geometric points of $S$ and isomorphism classes
    of geometric points of $\mSh^\stable_{C/k}(n,L)$;
  \item given an algebraically closed field $\kappa$ and two
    sheaves $H,H'\in\mSh^\semistable_{C/k}(n,L)(\kappa)$,
    we have that $\pi(H)=\pi(H')$ if and only if $\gr(H)\cong\gr(H')$.
  \end{enumerate}
\end{thm}
\begin{proof}[Extremely vague idea of proof]
  The very vague idea of the proof is this: given locally free sheaves $V$ on
  $C$, rigidify them by adding bases for the space $\Gamma(C,V(m))$ (i.e.,
  surjections $\ms O^{\chi(C, V(m))}\to V(m)$ for $m\gg 0$). One then shows that
  the rigidified problem lies inside an appropriate $\Quot$ scheme. This $\Quot$
  scheme has an action of $\PGL_N$, and Geometric Invariant Theory precisely
  describes quotients for such actions (once those actions have been suitably
  linearized in a projective embedding). The key to this part of
  the argument is the link between stability as defined above and the kind of
  stability that arises in GIT (which is phrased in terms of stabilizers of
  points).
\end{proof}

One of the most miraculous properties arising from stability is that
$\mSh^\stable_{C/k}(n,L)$ is separated. That is, a family of stable sheaves
parametrized by a dvr is unquely determined by its generic fiber. This is very
far from true in the absence of stability. For example, a generic extension of
$\ms O(-1)$ by $\ms O(1)$ on a curve $C$ (where $\ms O(1)$ is a sufficiently
ample invertible sheaf) is stable, but any
such sheaf admits the direct sum $\ms O(-1)\oplus\ms O(1)$ as a limit. (So, in
this case, a stable sheaf in a constant family admits a non-stable limit!)

We see from the preceding material that $\cSh^\semistable_{C/k}(n,L)$ is a
unirational projective variety, containing $\cSh^\stable_{C/k}(n,L)$ as a smooth
open subvariety that carries a canonical Brauer class represented by
$\mSh^\stable_{C/k}(n,L)$. This immediately gives us an interesting geometric picture.

\begin{thm}[Main result of \cite{MR2353678}]
  The Brauer group of $\cSh^\stable_{C/k}(n,L)$ is generated by the class
  $[\mSh^\stable_{C/k}(n,L)]$, whose order equals $\gcd(n, \deg(L))$.
\end{thm}

The most interesting open question about the variety $\cSh^\stable_{C/k}(n,L)$
is this: is it rational? The answer is known to be ``yes'' when $n$ and
$\deg(L)$ are coprime, but otherwise the question is mysterious.

\subsection{Sheaves on a surface}
\label{sec:surface}

Increasing the dimension of the ambient space vastly increases the complexity of
the moduli space of sheaves. The space is almost never smooth. Fixing discrete
invariants no longer guarantees irreducibility. The Kodaira dimension of the
moduli space can be arbitrary. One cannot always form locally free limits of
locally free sheaves. The book \cite{MR2665168} is a wonderful reference for the
geometry of these moduli spaces, especially their GIT quotients, for surfaces.
We touch on some of the main
themes in this section.

Fix a smooth projective surface $X$ over an algebraically closed field $k$; let
$\ms O(1)$ denote a fixed ample invertible sheaf. We first define the basic
stacks that contain the moduli problems of interest. Fix a positive integer $n$,
an invertible sheaf $L$, and an integer $c$.

\begin{defn}
 The \emph{stack of torsion free sheaves of rank $n$, determinant $L$, and
   second Chern class $c$\/},
 denoted $\mSh^\tf_{X/k}(n, L, c)$, is the stack whose objects over $T$ are
 pairs $(F,\phi)$, where $F$ is a 
 $T$-flat quasi-coherent sheaf $F$ of finite presentation, $\phi:\det(F)\simto
 L_T$ is an isomorphism, and such that for each
 geometric point $t\to T$, the fiber $F_t$ is torsion free of rank $n$
 and second Chern class $c$.

 There is an open substack $\mSh^\lf_{X/k}(n, L, c)$ parametrizing those sheaves
 that are also locally free.
\end{defn}

Since the degeneracy locus of a general map $\ms O^{n-1}\to\ms O^{n}$ has
codimension $2$, no analogue of the Serre trick will work to show that the stack
$\mSh^\lf_{X/k}(n, L, c)$ is unirational. In fact, this is not true.

The notion of stability is also significantly more complicated for a surface.
There are now several notions, and these notions depend upon a choice of ample
class. Given a sheaf $G$, let $P_G$ denote its Hilbert polynomial (whose value
at $n$ is $\chi(X,G(n))$). We can write
$$P_G(n)=\sum_{i=0}^d\alpha_i(G)\frac{n^i}{i!}$$ for some rational numbers
$\alpha$, where $d$ is the dimension of the support of $G$.

\begin{defn}
  Given a torsion free sheaf $F$ on a surface $X$, the \emph{slope\/} of $F$ is
  $$\mu(F)=\frac{\det(F)\cdot\ms O(1)}{\rk(F)},$$
  where $\cdot$ denotes the intersection product of invertible sheaves.

  Given any sheaf $G$ on $X$ with support of dimension $d$, the \emph{reduced
Hilbert polynomial of $G$\/}, denoted $p_G$ is the
polynomial $\frac{1}{\alpha_d}P_G$.
\end{defn}

Recall that there is an ordering on real-valued polynomials: $f<g$, resp.\
$f\leq g$, if and only if for all $n$ sufficiently large $f(n) < g(n)$, resp.\
$f(n)\leq g(n)$. We will use this to define one of the notions of stability for
sheaves on a surface.

\begin{defn}
  A torsion free sheaf $F$ on $X$ is \emph{slope-stable\/}, resp.\
  \emph{slope-semistable\/}, if for any non-zero subsheaf $G\subset F$ we have
  $\mu(G)<\mu(F)$, resp.\ $\mu(G)\leq\mu(F)$.

  A sheaf $G$ on $X$ is \emph{Gieseker-stable\/}, resp.\
  \emph{Gieseker-semistable\/}, if for any subsheaf $G\subset F$ we have
  $p_G < p_F$, resp.\ $p_G\leq p_F$.
\end{defn}
Slope-stability is often called $\mu$-stability, and similarly for
semistability. Gieseker-stability is usually simply called stability, and
similary for semistability.

The following lemma gives the basic relation between these stability conditions.
\begin{lem}
  Any slope-stable sheaf is Gieseker-stable. Any torsion free Gieseker-semistable sheaf is slope-semistable.
\end{lem}

Just as for curves, one has the following description of stable and semistable
moduli. We restrict our attention to torsion free sheaves for the sake of simplicity.
\begin{prop}
  The (semi-)stability conditions define a chain of open immersions of finite type
  Artin stacks over $k$
$$\mSh^{\mu-\stable}_{X/k}(n, L, c)\inj\mSh^{\stable}_{X/k}(n, L,
c)\inj\mSh^{\semistable}_{X/k}(n, L, c)\inj\mSh^{\mu-\semistable}_{X/k}(n, L,
c),$$
and all of these stacks are open in $\mSh^\tf_{X/k}(n, L, c).$
\end{prop}

Just as for sheaves on curves, a Gieseker-semistable sheaf $H$ admits a
Jordan-H\"older filtration into stable sheaves and a canonically associated
polystable sheaf $\gr(H)$.
\begin{thm}[Gieseker, modern interpretation]\label{sec:sheaves-surface}
  There is a diagram
  $$\xymatrix{\mSh^\stable_{X/k}(n, L, c)\ar[r]\ar[d]_\pi &
    \mSh^\semistable_{X/k}(n, L, c)\ar[d]^{\pi'}\\
    \cSh^\stable_{X/k}(n, L, c)\ar[r] & \cSh^\semistable_{X/k}(n, L, c),}$$
  in which
  \begin{enumerate}
  \item $\cSh^\semistable_{X/k}(n, L, c)$ is a projective scheme;
    \item $\cSh^\stable_{X/k}(n, L, c)\to \cSh^\semistable_{X/k}(n, L, c)$ is an open
      immersion;
      \item $\pi$ together with the scalar multiplication structure realize
  $\mSh^\stable_{X/k}(n, L, c)$ as a $\m_n$-gerbe over $\cSh^\stable_{X/k}(n, L, c)$.
  \end{enumerate}
  Moreover,
  \begin{enumerate}
  \item $\pi'$ universal among all morphisms from
    $\mSh^\semistable_{X/k}(n, L, c)$ to schemes;
  \item $\pi$ establishes a bijection between
    the geometric points of $S$ and isomorphism classes
    of geometric points of $\mSh^\stable_{X/k}(n, L, c)$;
  \item given an algebraically closed field $\kappa$ and two
    sheaves $H,H'\in\mSh^\semistable_{X/k}(n, L, c)(\kappa)$,
    we have that $\pi(H)=\pi(H')$ if and only if $\gr(H)\cong\gr(H')$.
  \end{enumerate}
\end{thm}
The reader will note that Theorem \ref{sec:sheaves-surface} is somewhat weaker
than Theorem \ref{sec:sheaves-curve}, in that it contains no information about the
geometry of $\cSh^\semistable_{X/k}(n,L,c)$. In fact, one should be careful to remember that the notion of stability depends
on the choice of $\ms O(1)$, and this changes the spaces that result in Theorem \ref{sec:sheaves-surface}. However, any two choices of
ample divisor class yield open substacks that are equal over dense open
substacks (so that generic properties are identical across polarizations), and
the resulting moduli spaces are birational.

\begin{example}
  Sometimes we can understand the moduli space very well. For example, if $X$ is
  a K3 surface and $n, L, c$ are discrete invariants such that the moduli space
  $\cSh^\stable_{X/k}(n, L, c)$ has dimension $2$, then any proper component of
  this space is itself a K3 surface, and one can calculate its cohomology
  explicitly. In particular, over $\C$, one can calculate its Hodge structure,
  which completely captures its isomorphism class (by the Torelli theorem).
\end{example}

The essence of the moduli theory of sheaves on surfaces, as discovered by
O'Grady, is in the hierarchy implicit in the second Chern class. If a moduli
space contains both locally free and non-locally free (but torsion free) points, then forming reflexive
hulls of torsion free points relates the boundary (the non-locally free locus)
to locally free loci in moduli spaces of \emph{lower\/} second Chern class. In
this way, one gets a hierarchy of moduli spaces in which the boundary of a given
space is made of spaces lower in the hierarchy. This is very similar to the
moduli space of curves: compactifying it, one finds that the boundary is made of
curves of lower genus, and this is what permitted Deligne and Mumford to give
their beautiful proof of the irreducibility of $\ms M_g$, essentially by
induction.

As O'Grady realized, a similar kind of ``induction'' gives us quite a bit of
information about the moduli spaces of sheaves on $X$, \emph{asymptotically in
  the second Chern class\/}. More precisely, he proved the following.

\begin{thm}[O'Grady]\label{sec:sheaves-surface-1}
  Given $n$ and $L$ and an integer $d$,
  there is an integer $C$ such that for all $c\geq C$, the
  projective scheme $\cSh^\semistable_{X/k}(n, L, c)$ is geometrically integral
  and lci with singular locus of codimension at least $d$.
\end{thm}

Numerous authors have studied these moduli spaces in special cases, their
Kodaira dimensions for surfaces of general type, etc. It is beyond the scope of
this article to give a comprehensive guide to literature, but the reader will be
well-served by consulting \cite{MR2665168} (which has as comprehensive a
reference list as I've ever seen).

\subsection{Guiding principles}
\label{sec:guide}

As a conclusion to this highly selective whirlwind tour of the basics of moduli
of semistable sheaves on curves and surfaces, let me state a few key points.

\begin{enumerate}
\item The geometry of moduli spaces of sheaves is not pathological for low-dimensional varieties.
\item In dimension $1$, the moduli spaces are rationally connected.
  \item In dimension $2$, a phenomenon very similar to the behavior of moduli of
    curves is observed: a hierarchy indexed by a discrete invariant that inductively improves geometry.
    \item (This one will be more evident below.) The classical moduli spaces
      (the coarse spaces or the GIT quotients) rarely carry universal sheaves,
      because the stable loci are often non-trivial gerbes over them.
  \end{enumerate}

Much of the recent successes achieved with twisted sheaves leverage these
phenomena, after recognizing that they continue to hold in greater generality
(for other kinds of ambient spaces, over more complex bases, without
stability conditions, on the level of stacks of sheaves, etc.).
As our thought experiment progresses in Section
\ref{sec:basics}, we will come back to these principles, establishing
appropriate analogues of them, and we will apply these analogues in Section
\ref{sec:case-studies} to give some concrete examples demonstrating how the
modicum of geometry we can develop in a more general setting suffices to prove
non-trivial results.

Before getting to that, however, we give a few examples showing foundational
difficulties inherent in this entire formulation of the moduli theory of sheaves.

\section{Some less classical examples}
\label{sec:motivations}

\subsection{A simple example}
\label{sec:motivating-example}

A simple example is given by the Picard scheme. Let's study a very special
case. Let $X$ be the conic curve in $\P^2_{\R}$ defined by the equation
$x^2+y^2+z^2=0$. This is a non-split conic. The Picard scheme of $X$ over
$\Spec\R$ is the constant group scheme $\Z_{\R}$. But not every section of $\Z$
corresponds to an invertible sheaf on $X$. For example, the section $1$ cannot
lift to an invertible sheaf, because any such sheaf would have degree $1$, and
thus $X$ would have an $\R$-point, which it does not.

How do we understand this failure? The classical answer goes as follows. Write
$$f:X\to\Spec\R$$ for the structure morphism. The Picard scheme represents the
functor $\R^1f_\ast\G_m$. The sequence of low-degree terms in the Leray
spectral sequence for $\G_m$ relative to $f$ is
$$0\to\Pic(X)\to\H^0(\Spec\R,\R^1f_\ast\G_m)\to\H^2(\Spec\R,\G_m)\to\H^2(X,\G_m).$$
The second group is just the space of sections of $\Pic_{X/\R}$, and this
sequence tells us that  
there is an obstruction in the Brauer group of $\R$ to lifting a section of the
Picard scheme to an invertible sheaf, and that this
obstruction lies in the kernel of the restriction map $\Br(\R)\to\Br(X)$. In the
case of the non-split conic, this restriction map is the $0$ map (i.e., the
quaternions are split by $X$), and $1\in\Z$ maps to the non-trivial element of
$\Br(\R)$, giving the obstruction map. The identity map
$\Pic_{X/\R}\to\Pic_{X/\R}$ is a section over $\Pic_{X/\R}$; there is an
associated \emph{universal obstruction\/} in $\H^2(\Pic_{X/\R},\G_m)$.

There is another way to understand this failure that is more geometric. Instead
of thinking about the Picard scheme, we could instead think about the Picard
\emph{stack\/}, which we will temporarily write as $\ms P$.
This is the stack on $\Spec\R$ whose objects over $T$ are
invertible sheaves on $X\times_{\Spec\R}T$. The classical formulation of
Grothendieck says precisely that the Picard scheme represents the
sheafification of $\ms P$. That is, there is a morphism $\ms P\to\Pic_{X/\R}$.
The stack $\ms P$ has a special property: the automorphism sheaf of any object of
$\ms P$ is canonically identified with $\G_m$ (i.e., scalar multiplications in a
sheaf). This makes $\ms P\to\Pic_{X/S}$ a \emph{$\G_m$-gerbe\/}. By Giraud's
theory, $\G_m$-gerbes over a space $Z$ are classified by $\H^2(Z,\G_m)$. The
cohomology class corresponding to $\ms P\to\Pic_{X/\R}$ is precisely the
universal obstruction -- the obstruction to lifting a point of $\Pic_{X/\R}$ to an
object of $\ms P$.

The advantage of the latter point of view is that we can then work with the
universal sheaf $\ms L$ on $X\times\ms P$. This universal sheaf has a special
property: the canonical right action of the inertia stack of $\ms P$ on $\ms L$ is
via the formula $(\ell,\alpha)\mapsto \alpha^{-1}\ell$ on local sections.
(Technical note: the inertial action is a \emph{right\/} action, whereas we
usually think of the $\ms O$-module structure on an invertible sheaf as a
\emph{left\/} action. We want the associated left action to be the usual scalar
multiplication action, which means that the right action must be the inverse of
this, resulting in the confusing sign.)

\subsection{A more complex example}
\label{sec:another-ex}

This example has to do with the Brauer group and Azumaya algebras. A reader
totally unfamiliar with these things is encouraged to come back to this example
later.

Fix an
Azumaya algebra $\ms A$ on a scheme $X$, say of degree $n$ (so that $\ms A$ has
rank $n^2$ as a sheaf of $\ms O_X$-modules). The sheaf $\ms A$ is a form of some
matrix algebra $\M_n(\ms O_X)$, where $n$ is a global section of the constant
sheaf $\Z$ on $X$. The Skolem-Noether theorem tells us that the sheaf of
automorphisms of $\M_n(\ms O_X)$ is $\PGL_n$, and descent theory gives us a
class in $\H^1(X,\PGL_n)$ that represents $\ms A$. The classical exact sequences 
$$1\to\G_m\to\GL_n\to\PGL_n\to 1$$
and
$$1\to\m_n\to\SL_n\to\PGL_n\to 1$$
yield a diagram of connecting maps
$$\xymatrix{& \H^2(X,\m_n)\ar[dd]\\
  \H^1(X,\PGL_n)\ar[ur]\ar[dr] & \\
  & \H^2(X,\G_m),}$$
giving us two abelian cohomology classes one can attach to $\ms A$. Write 
$[\ms A]$ for the image in $\H^2(X,\G_m)$. We know from the theory of Giraud that the
latter group parametrizes isomorphism classes of $\G_m$-gerbes $\ms X\to X$. 

When $[\ms A]=0$, the gerbe $\ms X$ is isomorphic to $\B\G_m\times X$, and the
algebra $\ms A$ is isomorphic to $\send(V)$ for some locally free sheaf $V$ on
$X$; moreover, $V$ is unique up to tensoring with an invertible sheaf on $X$.
This gives a classification of all Azumaya algebras with trivial Brauer class.
Moreover, this helps us understand their moduli: the moduli space is essentially
the moduli space of locally free sheaves on $X$ modulo the tensoring action of
the Picard scheme. When $X$ is a curve or surface, these spaces are rather
well understood; they have well-known geometrically irreducible locally closed
subspaces, etc.

But what about non-trivial classes? It turns out that there is always a locally
free sheaf $V$ \emph{on the gerbe $\ms X$\/} such that the original Azumaya
algebra $\ms A$ is isomorphic to $\send(V)$. This sheaf $V$ doesn't come from
$X$, but it has identical formal properties. In particular, we can study the
moduli of the Azumaya algebras $\ms A$ with the fixed class in terms of moduli
of the sheaves $V$ that appear. When $X$ is a curve or surface, we again get
strong structural results. In particular, there are always canonically defined
geometrically integral locally closed subschemes of the moduli space. This has
real consequences. For example, when working over a finite field, any geometrically integral
scheme of finite type has a $0$-cycle of degree $1$. If $X$ is a curve, the
moduli space is geometrically integral and rationally connected (if one fixes
the determinant, an operation that is quite mysterious from the Azumaya algebra
point of view), and then the Graber-Harris-Starr theorem shows that there is a
rational point. These geometric statements ultimately tell us that if $X$ is a
smooth projective surface over an algebraically closed or finite field and
$\alpha$ is a Brauer class on $X$ of order $n$, then there is an Azumaya algebra
on $X$ in the class $\alpha$ that has degree $n$.

The point of this example is to illustrate a particular workflow: an algebraic
question can be transformed into a geometric question about a moduli space of
sheaves on a gerbe. It turns out that if one places these gerbes and
associated moduli gerbes of sheaves on an equal footing, then the resulting symmetry in the theory
sheds light on several problems, coming from algebra and number theory. Among
these are the Tate conjecture for K3 surfaces, the index-reduction problem for
field extensions, and the finiteness of the $u$-invariant of a field of
transcendence degree $1$ over $\Q$. We will discuss some of these below.

\subsection{A stop-gap solution: twisted sheaves}
\label{sec:stopgap}

Faced with the kinds of difficulties described in the previous two
sections, algebraic geometers 
have spent the last decade or so rewriting the theory in terms of ``twisted
sheaves''. The basic idea is to embrace the gerbes that appear naturally in the moduli
problems (coming from their natural stacky structures). Given a variety $X$,
then, the moduli ``space'' of sheaves on $X$ is really a gerbe $\ms M\to M$ over
some other space (at least near the general point). The universal
sheaf lives on $X\times\ms M$ and can then be used to compare $X$ and
$\ms M$ (for example, using cohomology, $K$-theory,
Chow theory, derived categories, motives, \ldots).

In Section \ref{sec:catalog} we describe some of the results that this
approach has yielded, due to many authors
working in disparate areas of the subject over a number of years.

\section{A catalog of results}
\label{sec:catalog}

In this section, we provide a brief catalog of some results proven about twisted
sheaves and results that use twisted sheaves. The literature is rather vast, so
this is only the tip of the iceberg. The reader will profit from following the
references in the works described below. In particular, I have given very short
shrift to the massive amount of work done in the direction of mathematical
physics (including the deep work of 
Block, Pantev, Ben-Bassat, Sawon, and many others).

The ``theorems'' stated here are occasionally a bit vague. They should be seen
purely as signposts indicating an interesting paper that deserves a careful
reading. I have kept this catalog purposely impressionistic in the hope that a
reader might get intrigued by a theorem or two and follow the included
references.

\subsection{Categorical results}
\label{sec:structural}

\begin{thm}[Antieau, conjectured by C\u{a}ld\u{a}raru, \cite{MR3466552}]
  Given quasicompact quasiseparated schemes $X$ and $Y$ over a commutative ring
  $R$ and Brauer classes $\alpha\in\Br(X)$ and $\beta\in\Br(Y)$, if there is an
  equivalence $\Qcoh(X,\alpha)\cong\Qcoh(Y,\beta)$ then there is an isomorphism
  $f:X\to Y$ such that $f^\ast\beta=\alpha$.
\end{thm}

\begin{thm}[C\u{a}ld\u{a}raru, \cite{MR1887894}]
  Suppose $X\to S$ is a generic elliptic threefold with relative Jacobian $J\to
  S$, and let $\widebar J\to J$ be an analytic resolution of singularities. Let
  $\alpha\in\Br(\widebar J)$ be the universal obstruction to the Poincar\'e
  sheaf on $X\times_S J$. Then there is an equivalence of categories
  $\D(X)\cong\D(\widebar J, \alpha)$ induced by a twisted sheaf on
  $X\times_S\widebar J$. 
\end{thm}

\begin{thm}[Lieblich-Olsson, \cite{MR3429474}]
  In characteristic $p$, if $X$ is a K3 surfaces and $Y$ is a variety such that
  $\D(X)\cong\D(Y)$, then $Y$ is isomorphic to a moduli space of stable sheaves on $X$.
\end{thm}

\begin{thm}[Huybrechts-Stellari, \cite{MR2179782}]
  Given a K3 surface $X$ over $\C$ and a Brauer class $\alpha$, the set of isomorphism
  classes of pairs $(Y,\beta)$ with $Y$ a K3 surface and $\beta\in\Br(Y)$ such
  that $\D(X,\alpha)\cong\D(Y,\beta)$ is finite.
\end{thm}

\subsection{Results related to the geometry of
  moduli spaces}
\label{sec:fm-res}

\begin{thm}[Yoshioka, \cite{MR2306170}]\label{T:yosh}
  The moduli space of stable twisted sheaves on a complex K3 surface $X$ is an
  irreducible symplectic manifold deformation equivalent to a Hilbert scheme of
  points on $X$. 
\end{thm}

\begin{thm}[Lieblich, \cite{MR2309155}]
  The stack of semistable twisted sheaves with positive rank and fixed determinant on a curve is
  geometrically unirational. The stack of semistable twisted sheaves with
  positive rank, fixed determinant, and sufficiently large second Chern class on
  a surface is geometrically integral, and lci, with singular locus of high codimension. 
\end{thm}

\begin{thm}[Lieblich, \cite{1507.08387}]
  The Ogus moduli space of supersingular K3 surfaces is naturally covered by
  rational curves. Moreover, a general point is verifiably contained in
  countably many pairwise distinct images of $\A^1$.
\end{thm}

\subsection{Results related to non-commutative algebra}
\label{sec:non-comm-alg}

\begin{thm}[Gabber]
  If $X$ is a quasi-compact separated scheme with ample invertible sheaf then $\Br(X)=\Br'(X)$.
\end{thm}

\begin{thm}[de Jong, \cite{MR2060023}]
  If $K$ is a field of transcendence degree $2$ over an algebraically closed
  field, then for all $\alpha\in\Br(K)$, we have $\ind(\alpha)=\per(\alpha)$.
\end{thm}

\begin{thm}[Lieblich, \cite{MR3418522}]
  If $K$ is a field of transcendence degree $2$ over a finite field then for all
  $\alpha\in\Br(K)$ we have $\ind(\alpha)|\per(\alpha)^2$.
\end{thm}

\begin{thm}[Krashen-Lieblich, \cite{MR2428144}]
  Given a field $k$, a smooth proper geometrically connected curve $X$ over $k$,
  and a Brauer class $\beta\in\Br(k)$, the index of $\beta$ restricted to $k(X)$
  can be computed in terms of the restriction of the universal obstruction over
  the moduli space of stable vector bundles on $X$. 
\end{thm}

\subsection{Results related to arithmetic}
\label{sec:arith}

\begin{thm}[Lieblich-Parimala-Suresh, \cite{MR3263156}]
  If Colliot-Th\'el\`ene's conjecture on $0$-cycles of degree $1$ holds for
  geometrically rationally connected varieties, then any field $K$ of
  transcendence degree $1$ over a totally imaginary number field has finite $u$-invariant.
\end{thm}

\begin{thm}[Lieblich-Maulik-Snowden, \cite{MR3215924}]
  Given a finite field $k$, the Tate conjecture for K3 surfaces over finite
  extensions of $k$ is equivalent to the statement that for each finite
  extension $L$ of $k$, the set of isomorphism classes of K3 surfaces over $L$
  is finite.
\end{thm}

\begin{thm}[Charles, \cite{1407.0592}]
  The Tate conjecture holds for K3 surfaces over finite fields of characteristic at least $5$.
\end{thm}

\part{A thought experiment}\label{part:thought-experiment}
In this part, we describe a uniform theory that erases the geometric distinction
between the ambient space holding the sheaves and its moduli space:
the theory of \emph{{\simplespace}s}.

\section{Some terminology}
\label{sec:terminology}

We will work over the base site of schemes in the fppf topology. Given a stack
$Z$, we will let $\inertia(Z):\Inertia(Z)\to Z$ denote the inertia stack of $Z$,
with its canonical projection map. Given a morphism of stacks $p:X\to S$, we will
let $\Inertia(X/S)$ denote the kernel of the natural map $\Inertia(X)\to
p^\ast\Inertia(S)$. (Even though $p^\ast$ is not unique, the subsheaf
$\Inertia(X/S)\subset\Inertia(X)$ is uniquely determined by any choice of $p^\ast$.)
We will let $\Sh(Z)$ denote the
sheafification of $Z$, which is a sheaf on schemes. We will write $\Site(Z)$ for
the site of $Z$ induced by the fppf topology on the base category.

Several categories will be important throughout this paper. Given a stack $S$,
we will let $\Sch_S$ denote the category of schemes over $S$; this is the total
space of $S$, viewed as a fibered category over $\Spec\Z$. We will let
$\Set$ denote the category of sets and $\Groupoid$ denote the (2-)category of
groupoids. The notation $\GrpSch_S$ denotes the category of group schemes over a
base $S$.

We end this section by defining a convenient notational structure on the subgroups of the
multiplicative group. This will be useful in our discussion of {\simplespace}s
starting in Section \ref{sec:shpaces}.

\begin{defn}\label{defn:levels}
  Let $\Types$ be the category with objects $\N\cup\{\infty\}$, and with
  \begin{equation}
    \Hom(m,m') =
    \begin{cases}
      \{\emptyset\} & \text{if}\ m|m'\\
      \emptyset & \text{otherwise}
    \end{cases}
  \end{equation}
  By convention, we assume that every natural number divides $\infty$.
\end{defn}

\begin{notation}\label{notn:infty}
  The notation $\Z/\infty\Z$ will mean simply $\Z$.
\end{notation}

\begin{defn}\label{defn:level-group}
  Define the functor $$G:\Types\to\GrpSch_{\Z}$$ by letting $G(m)=\m_m$ for
  $m<\infty$ and $G(\infty)=\G_m$. Given a divisibility relation $m|m'$, the arrow $G(m)\to
  G(m')$ is the canonical inclusion morphism.
\end{defn}
 By Notation \ref{notn:infty}, for all $m$ in
  $\Types$ we have that $\Z/m\Z$ is naturally the dual group scheme of $G(m)$.

\subsection{The 2-category of {\simplespace}s}
\label{sec:shpaces}

In order to give a symmetric description of moduli problems attached to sheaves,
we will need to work with $\G_m$-gerbes and $\m_n$-gerbes. Rather than
constantly refer to ``$\G_m$-gerbes or $\m_n$-gerbes'', we will phrase results
using a theory we call \emph{{\simplespace}s\/}. We do this for
two reasons: (1) when people see the term ``$\G_m$-gerbe'' or ``$\m_n$-gerbe'',
they are primed to think of it as a \emph{relative\/} term (i.e., they think of
a gerbe over a space, as in Giraud's original theory), and (2) many of our
results work simultaneously for $\G_m$-gerbes and $\m_n$-gerbes, or involve
relationships among such gerbes, and we prefer to put them under a single
unified umbrella.

\begin{defn}
  A \emph{\simplespace\/} over a stack $S$ is a pair $(X\to S,i)$, where $X$ is
  a stack with a morphism to $S$ and
  $i:G(m)\to\Inertia(X/S)$ is a monomorphism of sheaves of groups with central
  image for some $m$. We will call $m$ the \emph{level\/} of the {\simplespace}. A morphism of {\simplespace}s
  $(X,i)\to(Y,j)$ is a $1$-morphism $f:X\to Y$ over $S$ such that the induced
  diagram of canonical maps
  $$\xymatrix{G(m)\ar[r]\ar[d] & \Inertia(X/S)\ar[d]\\
    f^\ast G(m')|_{Y}\ar[r] & f^\ast\Inertia(Y/S)}$$
  commutes. Here we mean that the level of $X$ (written $m$) is assumed to
divide the level of $Y$ (written $m'$), and the left vertical map is the one
induced by the canonical inclusion $G(m)\inj G(m')$.

  We will say that the {\simplespace} has a property of a
  morphism of stacks (e.g.\ ``separated'') if the underlying morphism $X\to S$ has
  that property. For example, if the underlying stack $X$ is relatively algebraic over $S$, we will call the
  {\simplespace} algebraic. 
\end{defn}

\begin{notation}
  We will generally abuse notation and omit $i$ from the notation, instead
  referring to $G(m)$ as a subsheaf of $\Inertia(X/S)$, and we will call it the
  \emph{level subgroup\/}. Given a merbe $X$, we will
  write $G(X)$ for the level subgroup when we do not want or need to refer
  explicitly to its level. We will write $\widebar X$ for the rigidification
  $X/G(X)$ in the sense of \cite{MR2427954,MR2007376,MR2450211}, and we will call it the \emph{rigidification of
    $X$\/}.
\end{notation}

As we know from [\emph{loc.\ cit.}], the morphism $X\to\widebar X$ is a $G(X)$-gerbe (in the
fppf topology on $\widebar X$).

\begin{notation}
  We will write $\SSp$ for the $2$-category of {\simplespace}s.
\end{notation}

Sending a {\simplespace} to its level defines a fibered category
$$\lambda:\SSp\to\Types.$$

\begin{lem}\label{lem:cocart}
  The functor $\lambda$ is co-Cartesian.
\end{lem}
\begin{proof}
  This means: given a {\simplespace} $X$ of level $m$ and a divisibility relation
  $m|m'$, there is a {\simplespace} $X'$ of level $m'$ with a map $X\to X'$ that
  is universal for morphisms from $X$ to {\simplespace}s of level $m'$. To see
  this, let $\widebar X$ be the rigidification of $X$ along $G(m)$. Giraud's
  construction \cite{MR0344253} gives rise to an
  extension of $X\to\widebar X$ to a $G(m')$-gerbe $X'\to\widebar X$. Any
  map of {\simplespace}s $X\to Y$ with $Y$ of level $m'$ induces a morphism of
  rigidifications $\widebar X\to\widebar Y$, giving a map $X\to Y\times_{\widebar
    Y}\widebar X$ of gerbes over $\widebar X$ that respects the map $G(m)\to
  G(m')$. The techniques of \cite{MR0344253} show that this factors uniquely through $X'$, as desired.
\end{proof}

\begin{example}\label{ex:can-space}
  Every morphism of stacks $X\to S$ has an associated {\simplespace} for each
  level $m$, namely
  $X\times\B G(m)$. Thus, for example, any morphism of schemes has a canonically
  associated {\simplespace} of level $m$.
\end{example}

\begin{defn}\label{defn:can-space}
  Given a {\simplespace} $X\to S$ of level $m$ and a divisibility relation
  $m|m'$,
  we will let $\Can_{m|m'}(X)$ denote the
  \emph{canonically associated {\simplespace} of level $m'$\/} arising from Lemma
  \ref{lem:cocart}.
\end{defn}
As an example, any stack is canonically a {\simplespace} of level $1$. For any
$m$, we have $$\Can_{1|m}(X)=X\times\B G(m),$$ with the inclusion
$G(m)\inj\Inertia(X\times\B G(m))$ arising from the canonical isomorphism
$$\Inertia(X\times\B G(m))\simto\Inertia(X)\times G(m)|_X.$$

\begin{remark}
  Given a merbe $X$ of level $m$ and two divisibility relations $m|m'|m''$,
  there is a canonical isomorphism
  $$\Can_{m'|m''}(\Can_{m|m'}(X))\simto\Can_{m|m''}(X)$$
  compatible with the natural maps from $X$.
\end{remark}

\begin{defn}
  A {\simplespace} $X$ is \emph{splittable\/} if there is an isomorphism
  $$\Can_{\lambda(X)|\infty}(X)\cong\Can_{1|\infty}(\widebar X).$$
\end{defn}
In less technical terms, a {\simplespace} is splittable if the associated
$\G_m$-gerbe is trivializable.

There is one more natural operation one can perform on {\simplespace}s.

\begin{defn}\label{def:contract}
  Given a {\simplespace} $X\to S$ and a divisibility relation $m|\lambda(X)$,
  the \emph{contraction\/} of $X$, denoted $\Contr_{\lambda(X)/m}(X)$ is the
  rigidification of $X$ along the subgroup $G(\lambda(X)/m)$.
\end{defn}
Thus, for example, we have $\widebar X=\Contr_{\lambda(X)/1}(X)$.

\subsection{Sheaves on {\simplespace}s}
\label{sec:sheaves-on-shpaces}

Fix a {\simplespace} $X\to S$. Since $G(m)$ is a subsheaf of $\Inertia(X/S)$
(and thus of $\Inertia(X)$), any abelian sheaf $F$ on $X$ admits a natural right
action of $G(m)$. In particular, if $X$ is algebraic and $F$ is quasi-coherent
$F$ breaks up as a direct sum of eigensheaves
$$F=\bigoplus_{n\in\Z/m\Z}F_n,$$
where the action $F_n\times{G(m)}\to F_n$ is described on local sections by
$$(f, \alpha)\mapsto\alpha^{-n}f.$$
(In other words, the left action canonically associated to the inertial action
is multiplication by $n$th powers via the $\ms O_X$-module structure on $F$.)

\begin{defn}
  A sheaf $F$ of $\ms O_X$-modules on the {\simplespace} $X$ will be called an
  \emph{$n$-sheaf\/} if the right action $F\times{G(m)}\to F$ satisfies
  $(f,\alpha)\mapsto\alpha^{-n} f$ on local sections.
\end{defn}

\begin{notation}
  We will write $\Sh^{(n)}(X)$ for the category of $n$-sheaves on $X$.
\end{notation}

In popular terminology, $n$-sheaves are usually called ``$n$-fold twisted
sheaves'' and $1$-sheaves are usually called ``twisted sheaves''. 

\begin{example}
  Given a morphism of stacks $X\to S$ with canonical {\simplespace} $\Can_{1|m}(X)$,
  there is a canonical invertible $n$-sheaf for every $n$. Indeed, via the canonical
  map $$\Can_{1|m}(X)\to\Can_{1|m}(S)=\B{G(m)}$$ we see that it suffices to demonstrate this for
  $\B{G(m)}$. Any character of ${G(m)}$ induces a canonical invertible sheaf
  (but note that the equivalence of representations of $G(m)$ and sheaves on $\B
  G(m)$ naturally involves \emph{right\/} representations, so the corresponding
  left action requires composition with inversion).
  The $n$th power map $$\xymatrix{{G(m)}\ar[r]^{\bullet^n} & {G(m)}\ar[r] & \G_m}$$ gives an $n$-sheaf. 
\end{example}

\begin{notation}
  The canonical invertible $n$-sheaf on $\Can(X)$ will be called $\chi^{(n)}_X$.
\end{notation}

The key to the symmetrization presented in this article is the
following pair of mundane lemmas.

\begin{lem}\label{lem:sheaves-simplespaces}
  Suppose $X\to S$ is a morphism of stacks. For any $n$ and $m$, there is a canonical
  equivalence of categories
  $$\Mod_{\ms O_X}\to \Sh^{(n)}(\Can_{1|m}(X))$$
  given by pulling back along the projection $\Can_{1|m}(X)\to X$ and tensoring with $\chi^{(n)}_X$.
\end{lem}
\begin{proof}
  The inverse equivalence is given by tensoring with $\chi^{(-n)}_X$ and pushing
  forward to $X$. More details that these are inverse equivalences may be found
  in \cite{MR2388554,MR2309155}. 
\end{proof}

\begin{lem}\label{L:level-transfer}
  For any {\simplespace} $X$ of level $m$ over $S$, any divisibility relation
  $m|m'$, and any $n$, the natural map $$X\to\Can_{m|m'}(X)$$ induces an
  equivalence $$\Sh^{(n)}(\Can_{m|m'}(X))\to\Sh^{(n)}(X)$$ by pullback. 
\end{lem}
\begin{proof}
  Let $\widebar X$ be the rigidification of $X$ along $G(m)$, so that
  $X\to\widebar X$ (respectively, $\Can_{m|m'}(X)\to\widebar X$) is a
  $G(m)$-gerbe (respectively, a $G(m')$-gerbe). Each of the the categories $\Sh^{(n)}(X)$
  and $\Sh^{(n)}(\Can_{m|m'}(X))$ is the global sections of a stack on $\widebar X$,
  and pullback is a morphism between these stacks. Working locally on $\widebar
  X$, we see that it thus suffices to prove this when $X=\widebar X\times\B
  G(m)$, in which case $\Can_{m|m'}(X)=\widebar X\times\B G(m')$. Since we are
  working on the big fppf site, and the stacks of
  $n$-sheaves are the restrictions of the corresponding stacks for $\B G(m)$ and
  $\B G(m')$ over $\Spec\Z$, we see that it suffices to prove the result for $\B
  G(m)$ and $\B G(m')$ over $\Z$. Now, the quasi-coherent sheaves parametrized by $T$
  are precisely right
  representations of $G(m)$ and $G(m')$ over $T$, and the $n$-sheaf condition just says
  that the module action factors through the map $G(m)\to\G_m\to\G_m$, where the
  latter is the $n$th power map. Both stacks are identified with
  the stack of $n$-sheaves on $\B\G_m$, completing the proof.
\end{proof}

\section{Moduli of sheaves: basics and examples}
\label{sec:basics}

\subsection{The basics}
\label{sec:basicsbasics}

In this section, we introduce the basic moduli problems. Fix an algebraic {\simplespace} 
$f:X\to S$. We will assume to start that $f$ is separated and of finite
presentation.

\begin{defn}
  Given a morphism of stacks $T\to S$, a \emph{family of sheaves on $X/S$
    parametrized by T\/} is a quasi-coherent sheaf $F$ on $X\times_X T$ that is
  locally of finite presentation and flat over $T$.
\end{defn}

We will be mostly interested in families of sheaves with proper support.

\begin{defn}
  A \emph{family of sheaves on $X/S$ with proper support parametrized by $T$\/} is a family of
  sheaves $F$ on $X/S$ parametrized by $T$ such that the support of $F$ is proper over $T$.
\end{defn}

\begin{defn}
  Given $n$, a \emph{family of $n$-sheaves\/} parametrized by $T\to S$ is a
  family of sheaves $F$ parametrized by $T$ such that $F$ is an $n$-sheaf.
\end{defn}

\begin{defn}
  A family $F$ of sheaves on $X/S$ parametrized by $T\to S$ is \emph{perfect\/}
  if there is an fppf cover $U=\{\sqcup U_i\}\to X$ by a disjoint union of
  affine schemes such that for each $i$ the restriction $F|_{U_i}$ has a finite resolution by
  free $\ms O_U$ modules.
\end{defn}

The key feature of a perfect sheaf $F$ is its determinant $\det(F)$, which is an
invertible sheaf on $X$ that is additive in $F$ (that is, given a sequence of
perfect sheaves $0\to F\to F'\to F''\to 0$ we have
$\det(F')\cong\det(F)\tensor\det(F'')$) and has value $\bigwedge^r F$ when $F$ is
locally free of constant rank $r$. The existence of the determinant
follows from the work of Mumford and Knudsen \cite{MR0437541}.

It is relatively straightforward to show that the pullback of a family is a
family, and that pullback respects perfection and determinants.

In the rest of the paper, I will relax the terminology and may refer more loosely to
objects of this kind as ``families of sheaves'' (without always mentioning the
original morphism or the base of the family).

The basic (almost totally vacuous) result about families of sheaves is the
following.

\begin{prop}\label{sec:basics-1}
  Given a {\simplespace} $X\to S$, families of $n$-sheaves on $X/S$ naturally
  form a {\simplespace} $\mSh_{X/S}^{(n)}\to S$ of level $\infty$. If $X$ is
  algebraic and of finite presentation over $S$, the perfect
  families are parametrized by an open substack $\mSh_{X/S}^{(n)\perf}\subset\mSh_{X/S}^{(n)}$.
\end{prop}
\begin{proof}
  We know that fppf descent is effective for
  quasi-coherent sheaves of finite presentation on a stack. Since the condition
  of being an $n$-sheaf is local in the fppf topology, descent is also effective
  for $n$-sheaves. Thus, $n$-sheaves form a stack $\mSh_{X/S}^{(n)}$. Finally, the
  module structure gives a canonical identification of $\G_m$ with a central
  subgroup of the inertia, making $\mSh_{X/S}^{(n)}$ a {\simplespace} of level
  $\infty$, as desired.

  To see the last statement, note that we may first assume $S$ is affine, and
  that $X$ admits an fppf cover by a finite list of affines $U_1,\ldots,U_\ell$, each of finite
  presentation over $S$. Suppose given an $n$-sheaf $F$ on $X_T$. We wish to
  establish that the locus in $T$ over which $F$ is perfect is open, and to do
  this we may work locally on $T$ and thus assume that $T$ is affine. By
  standard constuctibility results, this then readily reduces to the case in
  which $X$ itself is affine. For arguments of this kind, the reader is referred
  to \cite{MR0354655}, \cite{MR2177199}, or the beautiful appendix to
  \cite{MR1106918}.
\end{proof}

Write $\mSh_{X/S}^{(n)}(r)$ for the substack of sheaves of constant rank
$r\in\N$. Fix an invertible sheaf $L$ on $X$.

\begin{defn}
  The stack of perfect $n$-sheaves with determinant $L$ is the stack of pairs
  $(F,\phi)$, where $F$ is a perfect $n$-sheaf and $\phi:\det(F)\to L$ is an
  isomorphism. We will write it as $\mSh_{X/S}^{(n)\perf}(r,L).$
\end{defn}

\begin{lem}
  The stack $\mSh_{X/S}^{(n)\perf}(r,L)$ is naturally a {\simplespace} of level $r$.
\end{lem}
\begin{proof}
  The {\simplespace} structure arises from scalar multiplications as before, but since
  they have to commute with the identification of the determinant with $L$, only
  multiplications by sections of $\m_r$ give automorphisms.
\end{proof}

One of the great games of the moduli theory is to use the universal sheaf $F$ on
$X\times\mSh_{X/S}$ to push information between the two {\simplespace}s. We will
see examples of this later.

This theory is heavily driven by examples. We describe some of the basic
examples in the following sections using the language we have been developing here.

\subsection{Example: Almost Hilbert}
\label{sec:ex-hilb}

As students, we all learn to compute the
``Hilbert scheme of one point'' on $X/S$ and to show that it is $X$ itself (when
$X$ is separated!). After having done this, one 
invariably wonders what would happen were one to consider the ``Almost
Hilbert'' scheme of one point: colloquially, the scheme parametrizing sheaves on
$X$ of rank $1$ on closed points. We briefly investigate this
construction for merbes in this section as a first example.

\begin{assumption}\label{ass:almost-hilb}
  We assume throughout this section that $X\to S$ is an
  algebraic merbe that is locally of finite presentation whose rigidification $\widebar
  X\to S$ is a separated algebraic space.
\end{assumption}

\begin{defn}
  Given a merbe $X\to S$ satisfying Assumption \ref{ass:almost-hilb}, the
  \emph{almost Hilbert scheme of $X/S$\/}, denoted $\AHilb_{X/S}$ is the stack
  of families of $1$-sheaves on $X/S$ satisfying the following condition:
  for any family $F$ on $X\times_S T$ parametrized by $T$, we have that $F$ is
  an invertible $1$-sheaf supported on a closed sub{\simplespace} $Z\subset X$
  whose rigidification $\widebar Z\subset\widebar X$ is a section of the projection map
  $\widebar X\times_S T\to T$.
\end{defn}

\begin{prop}
  Given a merbe $X\to S$ satisfying Assumption \ref{ass:almost-hilb}, the stack
  $\AHilb_{X/S}$ is naturally a merbe of level $\infty$ and $$\AHilb_{X/S}\cong\Can_{\lambda(X)|\infty}(X).$$
\end{prop}
\begin{proof}
  Scalar multiplication defines the natural merbe structure on $\AHilb_{X/S}$;
  the level is $\infty$. Sending a family to its support defines an $S$-morphism
  $\AHilb_{X/S}\to\widebar X$. Fix a section $\sigma:T\to\widebar X$. The fiber
  $\AHilb_{X/S}\times_{\widebar X}T\to T$ is the $\G_m$-gerbe of invertible
  $1$-sheaves on $X\times_{\widebar X}T$, which is itself a $G(X)$-gerbe over
  $T$. As shown in \cite{MR2717173}, the $\G_m$-gerbe associated to $X\times_{\widebar X}T$
  is precisely the gerbe of invertible $1$-sheaves. (In the classical language,
  this is the gerbe of invertible twisted sheaves.) Applying this to the
  universal point $\id:\widebar X\to\widebar X$ establishes the desired isomorphism.
\end{proof}

\subsection{Example: invertible $1$-sheaves on an elliptic merbe}
\label{sec:genus-1}
Fix an elliptic curve $E$ over a field $K$. There are many merbes of level
$\infty$ with rigidification $E$. By Giraud's theorem, they are parametrized by
$\H^2(E,\G_m)$, which, by the Leray spectral sequence for $\G_m$, fits into a split exact
sequence
$$0\to\H^2(\Spec K,\G_m)\to\H^2(E,\G_m)\to\H^1(\Spec K,\Pic_{E/K})=\H^1(\Spec K,
\Jac(E))\to 0$$
(the latter equality arising from the vanishing of $\H^1(\Spec K,\Z)$). 
Among other things, this sequence tells us that any Brauer class on $E$ has an
associated $\Jac(E)$-torsor. How can we identify it?

Given a class $\alpha\in\Br(E)$ such that $\alpha|_{\bf 0}=0\in\Br(K)$ (with
$\bf 0$
denoting the identity point of $E$), let $\mc E$ be an associated merbe (in
classical notation, a $\G_m$-gerbe associated to $\alpha$). Let $\mc M$ be the stack
of invertible $1$-sheaves on $\mc E$; $\mc M$ is itself a merbe of level
$\infty$. Tensoring with invertible $0$-sheaves on $E$ defines an action of
$\Pic_{E/K}$ on $\mc M$. The resulting action on $\widebar{\mc M}$ gives a
torsor, and this is precisely the image of the coboundary map. 

\subsection{Example: sheaves on a curve}
\label{sec:ex-curve}

In this section we reprise a small chunk of the moduli theory of sheaves on a
curve.

\begin{defn}
  A {\simplespace} $C\to S$ is a \emph{smooth proper curve-{\simplespace}\/} (among {\simplespace}s) if its rigidification $\widebar
  C\to S$ is a smooth proper relative curve with connected geometric fibers.
\end{defn}

We start with a smooth proper curve-{\simplespace} $C\to S$. Fix a
positive integer $n$ and an invertible $n$-sheaf $L$ on $C$. We will 
consider some of the structure of $\mc M:=\Sh^{(1)}_{C/S}(n,L)^{\textrm{lf}}$,
where $\textrm{lf}$ stands for the locus parametrizing locally free sheaves.

\begin{thm}\label{T:exampl-sheav-curve}
  The merbe $\mc M\to S$ has type $n$, and its rigidification $\widebar{\mc M}$
  is an ascending union of open subspaces that have integral rational geometric fibers
  over $S$.
\end{thm}
\begin{proof}[Idea of proof]
  Tsen's theorem implies that there is an fppf cover $S'\to S$ such that
  $C\times_S S'$ is splittable. In this circumstance,
  by Lemma \ref{lem:sheaves-simplespaces}, there is an invertible sheaf $L'$ on
  $\widebar C\times_S S'$ and an isomorphism
  $$\Sh^{(1)}_{C/S}(L)\times_S S'\cong\Can_{1|n}(\Sh_{\widebar C\times_S
    S'/S'}(L')).$$
  We can then use the classical geometry of the stack of sheaves on a curve
  (described in Section \ref{sec:curve}) to
  deduce geometric properties of $\mc M\to S$.
\end{proof}

\begin{cor}\label{C:exampl-sheav-curve}
  If $C$ is a smooth proper curve-{\simplespace} over a finite extension $K$ of $k(t)$, with
  $k$ algebraically closed, and $L$ is an invertible $n$-sheaf with $n$ invertible
  in $K$, then
  $\mc M:=\Sh^{(1)}_{C/K}(n, L)^{\text{\rm lf}}$ contains an object over $K$. 
\end{cor}
\begin{proof}
  By Theorem \ref{T:exampl-sheav-curve}, $\mc M$ is a merbe of level $n$ whose
  rigidification is geometrically rationally connected. Since $n$ is invertible
  in $K$, the fiber of $\mc M$ is separably rationally connected. By
  \cite{MR1981034}, the rigidification of $\mc M$ has a $K$-rational point. But
  the obstruction to lifting from $\widebar{\mc M}$ to $\mc M$ lies in $\Br(K)$,
  which vanishes by Tsen's theorem.  
\end{proof}

\subsection{Example: sheaves on a surface}
\label{sec:ex-K3}

It turns out that essentially all of the key properties enumerated in Section
\ref{sec:guide} of classical theory of sheaves on surfaces
carry over to sheaves on smooth proper surface-{\simplespace}s. This has some
interesting consequences we will describe later.

\begin{defn}
  A \emph{smooth proper surface-{\simplespace}} is a {\simplespace} $X\to S$
  whose rigidification is a relative proper smooth surface with connected
  geometric fibers.
\end{defn}

The general theory is described in \cite{MR2309155} (using the language of twisted sheaves).
For the purposes of certain applications, let me isolate a much less precise
theorem that suffices in applications.

\begin{thm}\label{T:exampl-sheav-surf}
  Suppose $X\to\Spec k$ is a smooth proper surface-{\simplespace} over a field,
  $n$ is a natural number invertible in $k$, and $L$ is an invertible $n$-sheaf
  on $X$. Then the stack $\mc M:=\Sh^{(1)}_{X/k}(n,L)$ contains a geometrically
  integral locally closed substack, if $\mc M$ is non-empty.
\end{thm}
\begin{proof}[Idea of proof]
  The proof is rather complex, but the basic idea is easy to convey.
  The reader is referred to
\cite{MR3418522} for a proof in the language of twisted
sheaves. Roughly speaking, the idea is to organize a hierarchy of
  locally closed substacks of $\mc M$ using the second Chern class, and to show
  that the substacks get nicer and nicer as the second Chern class grows. Using
  deformation theory, one shows that the set of components must shrink as the
  second Chern class grows, and each component must get closer to being smooth.
  Eventually, all of the components coalesce and a geometrically integral locally
  closed substack results. The hierarchy in question arises from the
  taking the reflexive hull of a sheaf $\ms F$, and the ``height'' in the
  hierarchy is determined by the length of $\ms F^{\vee\vee}/\ms F$;
  sheaves with greater length have higher second Chern class.
\end{proof}

\begin{remark}
  The basic idea sketched above is a recurring theme in the theory of
  moduli. A propitious choice of compactification of a moduli problem
  leads to a boundary that can be understood using associated moduli
  spaces of a lower level in some natural hierarchy, and playing one
  level of the hierarchy off another leads to limiting theorems. The
  Deligne-Mumford proof of irreducibility of $\ms M_g$ \cite{MR0262240} is 
  essentially built from this idea: the boundary of $\widebar{\ms
    M}_g$ is stratified by pieces made out of lower-genus moduli
  spaces, and their argument essentially works by a subtle
  induction. Similarly, O'Grady's proof of the asymptotic
  irreducibility of the moduli space of stable vector bundles with
  fixed Chern classes on a smooth projective surface \cite{MR1376250} works by
  the same basic outline as above, using reflexive hulls to relate
  boundary strata of the moduli space to the open part of moduli
  spaces lower in the hierarchy. Finally, the work of de Jong, He, and
  Starr on higher rational connectedness \cite{MR2854858} uses this same strategy. The
  hierarchy in that case is indexed by degrees of stable maps into a
  fibration $f:X\to C$, and the transition among strata is achieved by attaching
  vertical curves to the image (thus inserting the sections of a given
  degree into the boundary of the compactified space of higher-degree sections).
\end{remark}

\begin{cor}
  Suppose $k$ is PAC. Given a proper smooth surface-{\simplespace} $X\to\Spec k$, a natural number
  $n$ invertible in $k$, and an invertible $n$-sheaf $L$, there is an object of
  of $\Sh^{(1)}_{X/k}(n, L)$ over $k$, assuming it is non-empty.
\end{cor}

One key example of a proper smooth surface-{\simplespace} comes from a
$\m_n$-gerbe over a root construction over a surface. As explained in
\cite{MR3418522}, the moduli space of $1$-sheaves on such {\simplespace}s are
key to controlling the period-index relation for Brauer classes over fields of
transcendence degree $2$ over a finite field.

\subsection{Example: sheaves on a K3 \simplespace}
\label{sec:K3-merbes}

While a lot more can be said for a general surface, the case of K3 surfaces is
a very interesting special case. The material in this section is drawn from
\cite{MR3215924} and \cite{1407.0592}. The latter contains a cleverer approach to
handling the numerical properties of the resulting theory, giving stronger
boundedness results (relevant in Section \ref{sec:derived}).

\begin{defn}
  A \emph{K3 {\simplespace}\/} is a proper smooth surface-{\simplespace} $X\to
  S$ such that each geometric fiber $X_s$ has rigidification isomorphic to a K3 surface.
\end{defn}

Fix a K3 {\simplespace} $X$ over a field $k$ and a prime $\ell$ that is
invertible in $k$. We assume for the sake of
simplicity that $\lambda(X)=m<\infty$ is invertible in $k$. In this case, since $X$ itself is smooth
and proper, the Chow theory $\Chow(X)$ and \'etale cohomology
$\H^\ast(X,\Z_\ell(i))$ behave well.

\begin{defn}
  The \emph{Mukai representation\/} attached to $X$ is the $\Gal_k$-module
  $$\H(X_{\widebar k},\Z_\ell):=\H^0(X_{\widebar k},\Q_\ell)\oplus\H^2(X_{\widebar k},
  \Q_\ell(1))\oplus\H^4(X_{\widebar k}, \Q_\ell(2)).$$
\end{defn}

There is a Galois-invariant quadratic form with values in $\Q_\ell$ coming
from the formula $(a,b,c)\cdot(a',b',c')=bb'-ac'-a'c$, with the products on the
right side arising from the usual cup product in \'etale cohomology.

Suppose that $X$ fits into a sequence of K3 {\simplespace}s $(X_n)$ with the following
properties.

\begin{enumerate}
\item $X_n$ has level $\ell^n$.
\item For each $n$, the {\simplespace} $X_n$ is isomorphic to
  $\Contr_{\ell^{n+1}/\ell^n}(X_{n+1})$ (in the notation of Definition \ref{def:contract}).
\end{enumerate}
For reasons that we will not go into here, such a sequence is the arithmetic
analogue of a $B$-field from mathematical physics.

In this case, for each $n$ one can find a distinguished lattice
$\Lambda_n\subset\H(X_n,\Q_\ell)$ together with an additive map $$v:K^{(1)}(X_n)\to\Lambda_n$$
such that for all $E,F\in K^{(1)}(X_n)$ we have $\chi(E,F)=-v(E)\cdot v(F)$.
Here, $K^{(1)}$ stands for the $K$-theory of $1$-sheaves. 
This is called the \emph{$\ell$-adic Mukai-Chow lattice of $X_n$\/}.

\begin{remark}
  It is an open question to understand the existence of the $\ell$-adic
  Mukai-Chow lattice in the absence of the full sequence $(X_n)$. This has
  numerical implications for the study of K3 {\simplespace}s over finite
  fields, and it is probably relevant to the existence of certain kinds of
  uniform bounds on the size of the transcendental quotient for Brauer groups of K3
  surfaces over finite fields and number fields.

  If $k$ is algebraically closed, then such a sequence always exists (for $\ell$
  invertible in $k$). 
\end{remark}

The fascinating part of the theory of K3 {\simplespace}s is that we can use the
Mukai-Chow lattice to manufacture new K3 {\simplespace}s. The key theorem is the
following. (The list of names is in roughly chronological order; first,
Mukai proved this for classical sheaves on K3 surfaces over $\C$, then Yoshioka
extended it to twisted sheaves on K3 surfaces over $\C$, and finally
Lieblich-Maulik-Snowden proved it for twisted sheaves on K3 surfaces over
arbitrary fields. For a
more detailed look at a larger class of moduli spaces, see \cite{1407.0592}.)

\begin{thm}[Mukai, Yoshioka, Lieblich-Maulik-Snowden]\label{T:cheese}
  Given $v\in\Lambda_n$ such that $\rk v=\ell^n$ and $v^2=0$ (in the lattice
  structure on $\Lambda_n$), the stack $M$ of stable $1$-sheaves $F$ on $X_n$ with
  $v(F)=v$ is a K3 {\simplespace} of level $\ell^n$. Moreover, the universal
  sheaf defines an equivalence of derived
  categories $$\D^{(-1)}(X_n)\cong\D^{(1)}(M).$$ Finally if there is some
  $u\in\Lambda_n$ such that $v\cdot u$ is relatively prime to $\ell$, then $M$
  is splittable.
\end{thm}

In other words, the numerical properties of the lattice $\Lambda_n$ can produce
equivalences of derived categories between various kinds of K3 {\simplespace}s,
and can also ensure that one side of each such equivalence is splittable. As we
will see below, this is directly relevant to the Tate conjecture for K3
surfaces. In this connection, we recall the following theorem.

\begin{thm}[Huybrechts-Stellari, \cite{MR2179782}]\label{T:yogurt}
  Give a K3 {\simplespace} $X$ over $\C$ of level $\infty$, there are only finitely many K3 {\simplespace}s
  $Y$ of level $\infty$ such that $\D^{(1)}(X)\cong\D^{(1)}(Y)$. 
\end{thm}
\begin{remark}
  By dualizing $X_n$ over its rigidification, the difference between $\D^{(1)}$
  and $\D^{(-1)}$ becomes unimportant. The restriction to level $\infty$ is
  necessary because passing from a {\simplespace} $Z$ to 
  $\Can_{\lambda(Z)|\lambda(Z)m}(Z)$ for any $m$ always induces an equivalence
  of abelian categories of $1$-sheaves, hence an equivalence of derived
  categories. So counting statements for derived equivalences should only be
  evaluated at level $\infty$.
\end{remark}

\section{Case studies}
\label{sec:case-studies}

We conclude this tour with a couple of almost immediate consequences of the
theory described above to two problems: the period-index problem for the Brauer
group and the Tate conjecture for K3 surfaces.

\subsection{Period-index results}
\label{sec:per-ind}

Let $K$ be a field. Given a Brauer class $\alpha\in\Br(K)$, there are two
natural numbers one can produce: the \emph{period\/} of $\alpha$, denoted
$\per(\alpha)$, and the \emph{index\/} of $\alpha$, denoted $\ind(\alpha)$.
Using basic Galois cohomology, one can show that $\per(\alpha)|\ind(\alpha)$ and
that both numbers have the set of prime factors. (Note: see \cite{MR3299728} for
the more interesting situation over a scheme larger than one point.) The basic
period-index problem is to determine how large $e$ must be in order to ensure
$\ind(\alpha)|\per(\alpha)^e$. A specific form of the basic question is due to Colliot-Th\'el\`ene.

\begin{question}[Colliot-Th\'el\`ene]\label{Q:ct}
  Suppose $K$ is a $C_d$-field. Is it always true that $\ind(\alpha)|\per(\alpha)^{d-1}$?
\end{question}

The key connection between this question and the theory developed here is the
following: given a Brauer class $\alpha$ over a field $K$, there is an
associated $\G_m$-gerbe $G\to\Spec K$. The index of $\alpha$ divides a number
$N$ if and only if there is a $1$-sheaf of rank $N$ on $G$. When $K$ is the
function field of a reasonable scheme $X$, the $\G_m$-gerbe $G$ extends to a
reasonable {\simplespace} closely related to $X$. By studying the stack of
$1$-sheaves, we get a new {\simplespace} whose rational-point properties are
closely bound to Question \ref{Q:ct}. (This is all explained in \cite{MR2388554}.)

For global function fields an affirmative answer is given by the
Albert-Brauer-Hasse-Noether theorem (see \cite{MR1501659} and  
\cite{MR1581351} for the classical references).
For function fields of surfaces over algebraically
closed fields, the conjecture was proven by de Jong \cite{MR2060023},
using the deformation theory of Azumaya algebras. There is also a
simple proof using the results of Section
\ref{sec:ex-curve}: fiber the surface $X$ over $\P^1$ and consider the generic
fiber. The Brauer class gives a curve-{\simplespace} $C\to\Spec k(t)$, and the associated
{\simplespace} of $1$-sheaves $M$ is filled with unirational opens. By Corollary
\ref{C:exampl-sheav-curve}, $M$ contains an object.

The first non-trivial class of
$C_3$-fields for which we know the answer is function fields of surfaces over
finite fields. It is an immediate consequence of the theory developed here.

\begin{thm}[Lieblich, \cite{MR3418522}]
  If $K$ has transcendence degree $2$ over a finite field, then
  $$\ind(\alpha)|\per(\alpha)^2$$ for all $\alpha\in\Br(K)$.
\end{thm}
\begin{proof}[Idea of proof]
  Given a class $\alpha$ of period $\ell$, one associates a smooth proper geometrically
  connected surface-{\simplespace} $X$ of level $\ell$ over a finite field. The
  {\simplespace} $M$ of $1$-sheaves of rank $\ell^2$ on $X$ contains a geometrically integral locally
  closed substack by Theorem \ref{T:exampl-sheav-surf} (a moderate amount of the
  work is devoted to non-emptiness, and is the explanation for the $\ell^2$).
  The Lang-Weil estimates then imply that there is an object of $M$ over a field
  extension prime to $\ell$. Standard Galois cohomology then tells us that the
  desired divisibility relation holds.
\end{proof}

\subsection{The Tate conjecture for K3 surfaces}
\label{sec:derived}

We conclude with a few words on the Tate conjecture for K3 surfaces. This gives
a somewhat better illustration of how the symmetric nature of the theory of
{\simplespace}s lets mathematical information flow.

\begin{thm}[Ogus, Nygaard, Maulik, Madapusi-Pera, Charles]
  The Tate conjecture holds for K3 surfaces over a finite field $k$ of
  characteristic at least $5$.
\end{thm}
\begin{proof}[Idea of proof]
  There are now many proofs of this in various forms. We comment on the proof of
  Charles \cite{1407.0592}, which builds on ideas developed in \cite{MR3215924}
  that have already appeared in Section \ref{sec:K3-merbes}. The idea is this:
  first, the Tate conjecture for all K3 surfaces over all finite extensions of
  $k$ is equivalent to the statement that for each such extension there are only
  finitely many K3 surfaces. Second, one can verify the finiteness statement for
  K3 surfaces using an analogue of Zarhin's trick, familiar from abelian varieties.

  This equivalence arises almost directly from Theorem \ref{T:cheese} and
  Theorem \ref{T:yogurt}. We sketch the implication of Tate from finiteness:
  If $X$ has infinite Brauer group, one gets a sequence
  of K3 {\simplespace}s $(X_n)$ as in Section \ref{sec:K3-merbes}, and thus by
  Theorem \ref{T:cheese} one gets a sequence of K3 surfaces $M_n$ arising as
  rigidications of {\simplespace}s of stable $1$-sheaves on each $X_n$, together
  with equivalences $\D^{(-1)}(X_n)\cong\D(M_n)$. But Theorem \ref{T:yogurt} says
  that for a fixed K3 surface, there are only finitely many such partners up to
  isomorphism. (One must do a little work to get this down from $\C$ to a
  relevant statement for our sequence of equivalences, but it can be done.) If
  there are only finitely many K3 surfaces, then we end up with infinitely many
  partners for one of them. The conclusion is that no such sequence $(X_n)$ can exist, so the
  Brauer group must be finite.

  To prove finiteness, Charles used inspiration from complex geometry
  (birational boundedness statements for holomorphic symplectic varieties) to
prove a Zarhin-type statement for K3 surfaces using moduli spaces of stable
twisted
  sheaves, which have properties in characteristic $p$ very similar to those
  described by Yoshioka over $\C$ \cite{MR2306170}, as hinted at in Theorem
  \ref{T:yosh}. (The proof is quite subtle, as Charles can't simply port over
  complex techniques to characteristic $p$; rather he needs to use a tiny chunk
  of the theory of canonical integral models of Shimura varieties and a relative
  Kuga-Satake map to get an appropriate replacement for a particular period map.)
\end{proof}

\bibliography{article}

\begin{thebibliography}{10}

\bibitem{MR0354655}
{\em Th\'eorie des intersections et th\'eor\`eme de {R}iemann-{R}och}.
\newblock Lecture Notes in Mathematics, Vol. 225. Springer-Verlag, Berlin-New
  York, 1971.
\newblock S{\'e}minaire de G{\'e}om{\'e}trie Alg{\'e}brique du Bois-Marie
  1966--1967 (SGA 6), Dirig{\'e} par P. Berthelot, A. Grothendieck et L.
  Illusie. Avec la collaboration de D. Ferrand, J. P. Jouanolou, O. Jussila, S.
  Kleiman, M. Raynaud et J. P. Serre.

\bibitem{MR2007376}
Dan Abramovich, Alessio Corti, and Angelo Vistoli.
\newblock Twisted bundles and admissible covers.
\newblock {\em Comm. Algebra}, 31(8):3547--3618, 2003.
\newblock Special issue in honor of Steven L. Kleiman.

\bibitem{MR2450211}
Dan Abramovich, Tom Graber, and Angelo Vistoli.
\newblock Gromov-{W}itten theory of {D}eligne-{M}umford stacks.
\newblock {\em Amer. J. Math.}, 130(5):1337--1398, 2008.

\bibitem{MR2427954}
Dan Abramovich, Martin Olsson, and Angelo Vistoli.
\newblock Tame stacks in positive characteristic.
\newblock {\em Ann. Inst. Fourier (Grenoble)}, 58(4):1057--1091, 2008.

\bibitem{MR1501659}
A.~Adrian Albert and Helmut Hasse.
\newblock A determination of all normal division algebras over an algebraic
  number field.
\newblock {\em Trans. Amer. Math. Soc.}, 34(3):722--726, 1932.

\bibitem{MR3466552}
Benjamin Antieau.
\newblock A reconstruction theorem for abelian categories of twisted sheaves.
\newblock {\em J. Reine Angew. Math.}, 712:175--188, 2016.

\bibitem{MR3299728}
Benjamin Antieau and Ben Williams.
\newblock The prime divisors of the period and index of a {B}rauer class.
\newblock {\em J. Pure Appl. Algebra}, 219(6):2218--2224, 2015.

\bibitem{MR2353678}
Vikraman Balaji, Indranil Biswas, Ofer Gabber, and Donihakkalu~S. Nagaraj.
\newblock Brauer obstruction for a universal vector bundle.
\newblock {\em C. R. Math. Acad. Sci. Paris}, 345(5):265--268, 2007.

\bibitem{MR1045822}
Siegfried Bosch, Werner L\"utkebohmert, and Michel Raynaud.
\newblock {\em N\'eron models}, volume~21 of {\em Ergebnisse der Mathematik und
  ihrer Grenzgebiete (3) [Results in Mathematics and Related Areas (3)]}.
\newblock Springer-Verlag, Berlin, 1990.

\bibitem{MR1581351}
R.~Brauer, E.~Noether, and H.~Hasse.
\newblock Beweis eines {H}auptsatzes in der {T}heorie der {A}lgebren.
\newblock {\em J. Reine Angew. Math.}, 167:399--404, 1932.

\bibitem{MR1887894}
Andrei C{\u{a}}ld{\u{a}}raru.
\newblock Derived categories of twisted sheaves on elliptic threefolds.
\newblock {\em J. Reine Angew. Math.}, 544:161--179, 2002.

\bibitem{1407.0592}
Fran\c{c}ois Charles.
\newblock Birational boundedness for holomorphic symplectic varieties, zarhin's
  trick for $k3$ surfaces, and the tate conjecture, 2014.

\bibitem{MR2060023}
A.~J. de~Jong.
\newblock The period-index problem for the {B}rauer group of an algebraic
  surface.
\newblock {\em Duke Math. J.}, 123(1):71--94, 2004.

\bibitem{MR2854858}
A.~J. de~Jong, Xuhua He, and Jason~Michael Starr.
\newblock Families of rationally simply connected varieties over surfaces and
  torsors for semisimple groups.
\newblock {\em Publ. Math. Inst. Hautes \'Etudes Sci.}, (114):1--85, 2011.

\bibitem{MR1981034}
A.~J. de~Jong and J.~Starr.
\newblock Every rationally connected variety over the function field of a curve
  has a rational point.
\newblock {\em Amer. J. Math.}, 125(3):567--580, 2003.

\bibitem{MR0262240}
P.~Deligne and D.~Mumford.
\newblock The irreducibility of the space of curves of given genus.
\newblock {\em Inst. Hautes \'Etudes Sci. Publ. Math.}, (36):75--109, 1969.

\bibitem{MR0344253}
Jean Giraud.
\newblock {\em Cohomologie non ab\'elienne}.
\newblock Springer-Verlag, Berlin-New York, 1971.
\newblock Die Grundlehren der mathematischen Wissenschaften, Band 179.

\bibitem{MR2665168}
Daniel Huybrechts and Manfred Lehn.
\newblock {\em The geometry of moduli spaces of sheaves}.
\newblock Cambridge Mathematical Library. Cambridge University Press,
  Cambridge, second edition, 2010.

\bibitem{MR2179782}
Daniel Huybrechts and Paolo Stellari.
\newblock Equivalences of twisted {$K3$} surfaces.
\newblock {\em Math. Ann.}, 332(4):901--936, 2005.

\bibitem{MR0437541}
Finn~Faye Knudsen and David Mumford.
\newblock The projectivity of the moduli space of stable curves. {I}.
  {P}reliminaries on ``det'' and ``{D}iv''.
\newblock {\em Math. Scand.}, 39(1):19--55, 1976.

\bibitem{MR2428144}
Daniel Krashen and Max Lieblich.
\newblock Index reduction for {B}rauer classes via stable sheaves.
\newblock {\em Int. Math. Res. Not. IMRN}, (8):Art. ID rnn010, 31, 2008.

\bibitem{MR2717173}
Max Lieblich.
\newblock {\em Moduli of twisted sheaves and generalized {A}zumaya algebras}.
\newblock ProQuest LLC, Ann Arbor, MI, 2004.
\newblock Thesis (Ph.D.)--Massachusetts Institute of Technology.

\bibitem{MR2177199}
Max Lieblich.
\newblock Moduli of complexes on a proper morphism.
\newblock {\em J. Algebraic Geom.}, 15(1):175--206, 2006.

\bibitem{MR2309155}
Max Lieblich.
\newblock Moduli of twisted sheaves.
\newblock {\em Duke Math. J.}, 138(1):23--118, 2007.

\bibitem{MR2388554}
Max Lieblich.
\newblock Twisted sheaves and the period-index problem.
\newblock {\em Compos. Math.}, 144(1):1--31, 2008.

\bibitem{MR3418522}
Max Lieblich.
\newblock The period-index problem for fields of transcendence degree 2.
\newblock {\em Ann. of Math. (2)}, 182(2):391--427, 2015.

\bibitem{1507.08387}
Max Lieblich.
\newblock Rational curves in the moduli of supersingular k3 surfaces, 2015.

\bibitem{MR3215924}
Max Lieblich, Davesh Maulik, and Andrew Snowden.
\newblock Finiteness of {K}3 surfaces and the {T}ate conjecture.
\newblock {\em Ann. Sci. \'Ec. Norm. Sup\'er. (4)}, 47(2):285--308, 2014.

\bibitem{MR3429474}
Max Lieblich and Martin Olsson.
\newblock Fourier-{M}ukai partners of {K}3 surfaces in positive characteristic.
\newblock {\em Ann. Sci. \'Ec. Norm. Sup\'er. (4)}, 48(5):1001--1033, 2015.

\bibitem{MR3263156}
Max Lieblich, R.~Parimala, and V.~Suresh.
\newblock Colliot-{T}helene's conjecture and finiteness of {$u$}-invariants.
\newblock {\em Math. Ann.}, 360(1-2):1--22, 2014.

\bibitem{MR1304906}
D.~Mumford, J.~Fogarty, and F.~Kirwan.
\newblock {\em Geometric invariant theory}, volume~34 of {\em Ergebnisse der
  Mathematik und ihrer Grenzgebiete (2) [Results in Mathematics and Related
  Areas (2)]}.
\newblock Springer-Verlag, Berlin, third edition, 1994.

\bibitem{MR0184252}
M.~S. Narasimhan and C.~S. Seshadri.
\newblock Stable and unitary vector bundles on a compact {R}iemann surface.
\newblock {\em Ann. of Math. (2)}, 82:540--567, 1965.

\bibitem{MR1376250}
Kieran~G. O'Grady.
\newblock Moduli of vector bundles on projective surfaces: some basic results.
\newblock {\em Invent. Math.}, 123(1):141--207, 1996.

\bibitem{MR1106918}
R.~W. Thomason and Thomas Trobaugh.
\newblock Higher algebraic {$K$}-theory of schemes and of derived categories.
\newblock In {\em The {G}rothendieck {F}estschrift, {V}ol.\ {III}}, volume~88
  of {\em Progr. Math.}, pages 247--435. Birkh\"auser Boston, Boston, MA, 1990.

\bibitem{MR2306170}
K{\=o}ta Yoshioka.
\newblock Moduli spaces of twisted sheaves on a projective variety.
\newblock In {\em Moduli spaces and arithmetic geometry}, volume~45 of {\em
  Adv. Stud. Pure Math.}, pages 1--30. Math. Soc. Japan, Tokyo, 2006.

\end{thebibliography}
\bibliographystyle{plain}

\end{document}